\documentclass[reqno,a4paper,11pt]{amsart}

\usepackage{amsfonts,amsmath,amsthm,graphicx}

\newtheorem{thm}{Theorem}
\newtheorem{lem}[thm]{Lemma}

\newtheorem{cor}[thm]{Corollary}
\newtheorem{prop}[thm]{Proposition}
\theoremstyle{definition}
\newtheorem{rem}[thm]{Remark}
\newcommand{\ind}{\hspace{4ex}}

\begin{document}

\title[Majority Bootstrap Percolation on $G(n,p)$]{Majority Bootstrap Percolation on $G(n,p)$}

\date{\today}

\date{\today}

\author[C.~Holmgren]{Cecilia Holmgren$^\natural$}
\address{$^\natural$ Department of Mathematics, Uppsala University, SE-75310 Uppsala, Sweden and Department of Pure Mathematics and Mathematical Statistics, University of Cambridge, Cambridge CB3 0WB, UK.. %Ph: {\tt +46~8~16~4564}
 Supported by the Swedish Research Council.}
\email{cecilia.holmgren@math.uu.se}

\author[T.~Ju\v{s}kevi\v{c}ius]{Tomas Ju\v{s}kevi\v{c}ius$^\dag$}
\address{$^\dag$Department of Mathematical Sciences, University of Memphis, Memphis, TN 38152 USA. %Ph: {\tt +1~901~678~1319}, Fax: {\tt +1~901~678~4481}
}
\email{tomas.juskevicius@gmail.com}

\author[N.~Kettle]{Nathan Kettle$^\ddag$}
\address{$^\ddag$Department of Pure Mathematics and Mathematical Statistics, University of Cambridge, Cambridge CB3 0WB, UK. %Ph: {\tt +44~1223~7~64253}
}
\email{nathan.kettle@cantab.net}

\begin{abstract}
Majority bootstrap percolation on a graph $G$ is an epidemic process defined in the following manner. Firstly, an initially infected set of vertices is selected. Then step by step the vertices that have more infected than non-infected neighbours are infected. We say that percolation occurs if eventually all vertices in $G$ become infected.\\

In this paper we study majority bootstrap percolation on the Erd\H{o}s-R\'{e}nyi random graph $G(n,p)$ above the connectivity threshold. Perhaps surprisingly, the results obtained for small $p$  are comparable to the results for the hypercube obtained by  Balogh, Bollob\'as and Morris \cite{balogh2009majority}.
\end{abstract}



\maketitle

\section{Introduction}

\ind The classical bootstrap percolation, called $r$-neighbour bootstrap percolation, concerns a deterministic process on a graph. Firstly, a subset of the vertices of a graph $G$ is initially infected. Then at each time step the infection spreads to any vertex with at least $r$ infected neighbours. This process is a cellular automaton, of the type first introduced by von Neumann in \cite{von1966theory}. This particular model was introduced by Chalupa, Leith and Reich in \cite{chalupa1979bootstrap}, where $G$ was taken to be the Bethe lattice.

\ind A standard way of choosing the initially infected vertices is to independently infect each vertex with probability $p$. The probability that the entire graph eventually becomes infected is increasing with $p$. It is therefore sensible to study the quantity $p_{c}=\inf\{p:\mathbb{P}_p(G\mbox{ infected})\geq c\},$ in particular the critical probability $p_{1/2}$ and the size of the critical window $p_{1-\epsilon}-p_{\epsilon}$.

\ind A natural setting for this problem is the finite grid $[n]^d$. Many of the results on bootstrap percolation concern this problem. The first to study this graph were Aizenman and Lebowitz in \cite{aizenman1988metastability}, who showed that in 2-neighbour bootstrap percolation when $d$ is fixed we have $p_{1/2}=\Theta ((\log n)^{1-d})$. 

\ind The $r$-neighbour bootstrap percolation process has also been studied on the random regular graph by Balogh in \cite{balogh2007bootstrap} and on the Erd\H{o}s-R\'{e}nyi random graph $G(n,p)$ by Janson, {\L}uczak, Turova and Vallier in \cite{janson2012bootstrap}.

\ind In majority bootstrap percolation a vertex becomes infected if a majority of its neighbours are. In \cite{balogh2009majority} Balogh, Bollob\'as and Morris studied this process on the hypercube and showed that if the vertices of the $n$-dimensional hypercube are independently infected with probability 
$$q=\frac{1}{2}-\frac{1}{2}\sqrt{\frac{\log n}{n}}+\frac{\lambda\log\log n}{\sqrt{n\log n}},$$
then, with high probability, percolation occurs (i.e., all vertices eventually become infected) if $\lambda>\frac{1}{2}$ and does not occur if $\lambda\leq-2$. 

\ind In this paper we shall study majority bootstrap percolation on the Erd\H{o}s-R\'{e}nyi random graph $G(n,p)$ above the connectivity threshold.
We will see that for small $ p $ our results are in fact comparable to the results for the hypercube in \cite{balogh2009majority}, noting that the degree for each vertex in the $ n $-dimensional hypercube (with $ 2^{n} $ vertices) is equal to $ n $.

\section{Main Results} 

\ind In this section we shall state our main results and discuss two different ways of selecting the initially infected set. 
The proofs of these theorems (in Section \ref{upper} and Section \ref{lower}) use inequalities that are described separately in Section \ref{ineqineq}. %so as not to disrupt the flow of our arguments.

\ind For  a graph $G$ with some subset $I_0\subset V(G)$ of initially infected vertices, the majority bootstrap process on $G$ is defined by setting $I_{t+1}=I_t\cup\{v\in V(G):|I_t \cap \Gamma(v)|\geq\frac{|\Gamma(v)|}{2}\},$ where $ \Gamma(v) $ is the neighbourhood of $ v $. For a finite graph $G,$ this process will terminate with $I_{T+1}=I_T.$ Denote by $I=I_T$ the set of eventually infected vertices.  

\ind We shall look at the case of $G=G(n,p),$ the graph on $n$ vertices, where each edge is included independently with probability $p$. Often $ p:=p(n)\rightarrow \infty$ as $ n\rightarrow \infty $, but we use the standard notation to just write $ p $ also for functions depending on $ n $.
 Our initial setup is slightly different than for the hypercube mentioned above, instead of infecting each vertex independently with some probability $q$, we shall infect a random set of vertices of size $m:=m(n)$.

\ind In the normal setup for the majority bootstrap process on $G(n,p)$, we would first choose the edges of $G(n,p)$, and then choose an initially infected set $I_0$ uniformly from $[n]^{(m)}.$ As these two choices are independent we shall equivalently set $I_0=[m]$, and then choose the edges of $G(n,p)$. This is the $MB(n,p\ ;m)$ process.

\ind We now introduce some notation that shall be used. We use the standard asymptotic little-$o$ notation and this is always taken as $n$ or $N$ tends to infinity, i.e., if $ (b_n) $ is a sequence of numbers, we say that $ b_n=o(a_n) $ if $ b_n/a_n\rightarrow 0 $, as $ n\rightarrow \infty $.% in this chapter. 
   We set $d=\frac{np}{1-p}$, thus $d$ is roughly the average degree in $G(n,p)$ for $p=o(1)$.  We denote the binomial distribution with parameters $n$ and $p$ by $B(n,p)$. We shall sometimes abuse the notation and denote by $B(n,p)$ a random variable that has a binomial distribution. We reserve $m$ for the size of $I_0$ and shall always assume that   $$m=\frac{n}{2}-\frac{n}{2}\sqrt{\frac{\log d}{d}}+\lambda n\frac{\log \log \log d}{\sqrt{d\log d}}+o\left(n\frac{\log \log \log d}{\sqrt{d\log d}}\right),$$ for some constant $\lambda$.  We also use the standard notation that an event $ E_n $ holds \emph{with high probability}, i.e., for the event $ E_n $ it holds that $\mathbb{P}(E_n)\rightarrow 1$, as $ n\rightarrow \infty $.
Let $\omega(n)$ denote some arbitrary positive function that is increasing and unbounded, as $ n $ tends to infinity.% Unless otherwise stated any random variables mentioned will be independent.% Throughout this chapter 
  
  The inequalities below are only claimed to be true for $n$ large enough. For the $MB(n,p\ ;m)$ process, define 
$$\mathcal{P}_{m}\left(G(n,p)\right)=\mathbb{P}\left(I=[n]\right).$$

\ind We shall now state the main result of this paper.

\begin{thm}
\label{MajorBoot:mainthm}
Fix some number $\epsilon>0$. Assume that for $ n $ large enough,
\begin{equation*}
(1+\epsilon)\log n\leq p(1-p)n. 
\end{equation*}
If the initially infected set $I_{0}$ has size 
\begin{equation*} %\label{sizeinitialinfected}
m=\frac{n}{2}-\frac{n}{2}\sqrt{\frac{\log d}{d}}+\lambda n\frac{\log \log \log d}{\sqrt{d\log d}}+o\left(n\frac{\log \log \log d}{\sqrt{d\log d}}\right),
\end{equation*}
then
\begin{equation*}
\qquad \mathcal{P}_m\left(G(n,p)\right) \stackrel{n\rightarrow \infty}{\longrightarrow}
\begin{cases}
1, \,\,\, \text{if}\,\,\, \lambda > \frac{1}{2}, \\
0, \,\,\,  \text{if}\,\,\, \lambda <0.
\end{cases}
\end{equation*}

\end{thm}
%\begin{rem} 
%In Theorem \ref{MajorBoot:mainthm} we could assume that $ p<1-\delta $, where $ \delta>0 $ is an arbitrarily small constant. For simplicity we just assume that $ p\leq 0.99 $. Also in Corollary \ref{MajorBoot:maincor} below we could replace $ p\leq 0.99 $ with $ p<1-\delta $.
%\end{rem}

\ind Our second result concerns a more natural setup, where each vertex is initially independently infected with probability $q$, we have that, with high probability, $||I_0|-qn|\leq \omega(n)\sqrt{q(1-q)n}$. When $\sqrt{n}\ll \frac{n\log\log\log d}{\sqrt{d\log d}},$ i.e, when $p\ll \frac{(\log\log\log n)^2}{\log n}$, our result above shall still hold in this setting for $q=m/n$.

\ind More formally define the $MB^\prime(n,p\ ;q)$ to be the process in which the graph $G(n,p)$ is chosen, and each vertex is initially infected independently with probability $q$. Then the infection spreads by the majority bootstrap percolation process. For the process $MB^\prime(n,p\ ;q)$ define
$$\mathcal{P}^\prime_q(G(n,p))=\mathbb{P}(I=[n]).$$

\begin{cor}
\label{MajorBoot:maincor}
Fix some number $\epsilon>0$. Assume that  for $ n $ large enough,
$$(1+\epsilon)\log n\leq p(1-p)n.$$

 If  $p\ll \frac{(\log\log\log n)^2}{\log n}$, then with   $q=\frac{1}{2}-\frac{1}{2}\sqrt{\frac{\log d}{d}}+\lambda\frac{\log\log\log d}{\sqrt{d\log d}}$, we have
\begin{equation*}
\qquad \mathcal{P}^\prime_q\left(G(n,p)\right)  \stackrel{n\rightarrow \infty}{\longrightarrow}
\begin{cases}
1, \,\,\, \text{if}\,\,\, \lambda > \frac{1}{2}, \\
0, \,\,\,  \text{if}\,\,\, \lambda <0.
\end{cases}
\end{equation*}
If  $p\gg \frac{(\log\log\log n)^2}{\log n}$, then with $q=\frac{1}{2}-\frac{1}{2}\sqrt{\frac{\log d}{d}}+\theta\frac{1}{\sqrt{n}}$, we have
$$\mathcal{P}^\prime_q\left(G(n,p)\right)\rightarrow\Phi(2\theta),$$
where $\Phi(x)$ denotes the distribution function of the standard Normal random variable.

\end{cor}

\begin{proof}
As each vertex is infected independently, $|I_0|$ has distribution $B(n,q)$. Thus, with high probability, it holds that $||I_0|-qn|\leq\omega(n)\sqrt{q(1-q)n}.$ If $p\ll\frac{(\log\log\log n)^2}{\log n}$, then $n\frac{\log\log\log d}{\sqrt{d\log d}}\gg \sqrt{n}$ and the result follows from Theorem~\ref{MajorBoot:mainthm}. 

\ind If $p\gg\frac{(\log\log\log n)^2}{n}$, then for each fixed $\delta>0$ by the Central Limit Theorem we obtain
\begin{align*}
\mathcal{P}^\prime_q\left(G(n,p)\right)
&=\sum_{m=0}^n\mathbb{P}\left(B(n,q)=m\right)\mathcal{P}_m\left(G(n,p)\right)\\
&\geq\mathbb{P}\left(B(n,q)\geq qn+(\delta-\theta)\sqrt{n}\right)\mathcal{P}_{\lfloor qn+(\delta-\theta){\sqrt{n}}\rfloor}\left(G(n,p)\right)\\
&=\mathbb{P}\left(B(n,q)/\sqrt{q(1-q)n}\geq (qn+(\delta-\theta)\sqrt{n})/\sqrt{q(1-q)n}\right)(1+o(1))\\
&\rightarrow\Phi(2(\theta-\delta)),
\end{align*} 
where the fourth line follows as $\mathcal{P}_{\lfloor qn+(\delta-\theta){\sqrt{n}}\rfloor}\left(G(n,p)\right)\rightarrow 1$ for $p\gg\frac{(\log\log\log n)^2}{\log n}$ by Theorem \ref{MajorBoot:mainthm}. A similar argument shows that $$1-\mathcal{P}^\prime_q(G(n,p))\geq\Phi(-2(\theta+\epsilon))(1+o(1)),$$ 
and so 
$$\mathcal{P}^\prime_q(G(n,p))\rightarrow\Phi(2\theta).$$ 
\end{proof}

\ind When $p$ is smaller than the connectivity threshold, $G(n,p)$ contains isolated vertices. Due to the way we define the $MB(n,p\ ;m)$ process, any uninfected isolated vertex becomes infected in the first time step, so this is not an obstruction to complete percolation. However, once $p$ drops to below $\frac{\log n}{2n},$ then, with high probability, $G(n,p)$ contains isolated edges and neither endpoint of an isolated edge becomes infected if both endpoints are initially uninfected. This means that $\mathcal{P}_m(G(n,p))\rightarrow 0$ unless $m=n-o(n)$.   

%Theorem~\ref{MajorBoot:mainthm} is only stated for $p$ some way above the connectivity threshold. This is because for $p<\frac{(1-\epsilon)\log n}{2n}$, with high probability $G(n,p)$ contains connected components that consist of a single edge. A single edge will not be infected with high probability until $m=n-o(n)$, and so the critical probability for the entire graph to percolate is dominated by how the process works on these small components. 

%A more natural question in this range is does the giant component percolate. We only require $p=\omega(n)\log n/n$ in Lemma~\ref{MajorBoot:nolargenasty}.
\begin{rem} 
%The results of this paper were established at Cambridge University and Memphis University. 
Preliminary versions of this paper (including the same results) were  included in the Phd thesis by Kettle \cite{kettle2014thesis} and in the PhD thesis by Ju\v{s}kevi\v{c}ius \cite{juskevicius2015thesis}. There is also a recent study by Stef\'ansson and Vallier \cite{stefanssonvallier} on this subject using completely different methods than those that are used in this paper (but using similar methods as was used by Janson, {\L}uczak, Turova and Vallier in \cite{janson2012bootstrap}), where they show the first asymptotics of the thresholds $ m\sim\frac{n}{2} $ in Theorem   \ref{MajorBoot:mainthm} above,
and similarly thus the first asymptotics of the threshold $ q\sim \frac{1}{2} $ in Corollary \ref{MajorBoot:maincor} above.
\end{rem}
 
\section{Upper Bound}\label{upper}

\ind As $G$ is finite the $MB(n,p\ ;m)$ process will eventually terminate with some set $I\subset [n]$ of infected vertices. If we do not infect the whole graph, or, equivalently, we have that $I\neq [n]$, then we can say something about the structure of $I$. We shall call a proper subset $S$ of $[n]$ \emph{closed} if for all $v\in [n]\setminus S$ we have $|\Gamma(v)\cap S|<\frac{\left|\Gamma(v)\right|}{2}.$
Recall that $ I_0 $ is the set of initial infected vertices and that a vertex $ v\in I_{t+1} $, if either $ v\in I_t $ or if at least half of its neighbours lies in $ I_t $. In particular $ I_t\subseteq I_{t+1} $. 
If the majority bootstrap process does not percolate, let $ T $ be such that the process has stabilized, i.e., $I=I_{T}=I_{T+1}\neq [n]$. Then $ I $ is a closed set, and thus we must have that the initially infected vertices $I_0$ is a subset of a closed set. We shall show that, if $ \lambda>\frac{1}{2} $, then, with high probability, $I_0$ is contained in no closed sets in three stages. Using Lemma~\ref{MajorBoot:nolargenasty} will allow us that, with high probability, the graph $G(n,p)$ has no "large" closed sets. After that we shall bound the expected number of medium sized closed sets that $I_0$ is contained in, hence by the Markov inequality it
will follow that, with high probability, there are no medium sized closed sets containing $ I_0 $. But before we proceed with proving these two facts, we shall show that, with high probability, the number of infected vertices after one time step, $|I_1|,$ is large, and so $I_0$ can rarely be contained in a small closed set. Recall that
$$m=\frac{n}{2}-\frac{n}{2}\sqrt{\frac{\log d}{d}}+\lambda n\frac{\log \log \log d}{\sqrt{d\log d}}+o\left(n\frac{\log \log \log d}{\sqrt{d\log d}}\right).$$
We assume that for some fixed $ \epsilon>0 $ it holds that for $ n $ large enough $(1+\epsilon)\log n\leq p(1-p)n$. However, for some of the results below it is enough to assume that $ p(1-p)n=\omega(n) $.

\begin{lem}
\label{MajorBoot:upbstep1}
In the $MB(n,p\ ;m)$ process,
$$|I_1\setminus I_0|\geq \frac{n(\log\log d)^{2\lambda}}{e^8\sqrt{d \log d}},$$ 
with high probability.
\end{lem}

\begin{proof}

For $i\in \left[n\right]\setminus I_0$, denote by $A_i$ the event that vertex $i$ is infected at time one, that is the event that $i$ has fewer neighbours in $[n]\setminus I_0$ than it does in $I_0$. The events $A_i$ are identical and very weakly correlated but not independent. Let $X$ be the number of vertices infected at the first step of the process. Then $X=|I_1\setminus I_0|=\sum\mathbf{1}(A_i)$. 
We shall use Chebyshev's inequality to bound the probability that $X$ is small. 

\ind As the events $A_i$ are identical we shall set $r=\mathbb{P}(A_i)$, so $\mathbb{E}(X)=(n-m)r$. Let $B(m,p)$ and $B(n-m-1,(1-p))$ be independent random variables with means $\mu_1$ and $\mu_2$, respectively. We have that 
\begin{align*}
r&=\mathbb{P}\left(\left|\Gamma(i)\cap I_0\right|\geq\Gamma(i)\cap ([n]\setminus I_0)\right)\\
&=\mathbb{P}\left(B(m,p)\geq B(n-m-1,p)\right)\nonumber\\
&=\mathbb{P}\left(B(m,p)+B(n{-}m{-}1,(1{-}p))\geq \mu_1+\mu_2+p(n{-}2m{-}1)\right).
\end{align*}
 
\ind For $ p\gg \frac{1}{n} $, we have $p(n-2m-1)=\omega(n)\sqrt{p(1-p)n}$ and $p(n-2m-1)^2=o(n\sqrt{p(1-p)n})$. Applying the bound from Proposition~\ref{MajorBoot:twobinomlower} to the last equality with $N=\frac{n-1}{2}$, $S=\frac{n-1-2m}{2}$ and $h=p(n-2m-1)$, we obtain
\begin{align}
\label{MajorBoot:rbound}
r&>\frac{\sqrt{p(1-p)(n-1)}}{2\pi p(n-2m-1)}\exp\left(-\frac{p(n-2m-1)^2}{2(1-p)(n-1)}-4-o(1)\right)\nonumber \\
&>\frac{1}{2\pi\sqrt{\log d}}\exp\left(-\frac{\log d}{2}+2\lambda\log\log\log d-4+o(1)\right)\nonumber \\
&>\left(\frac{(\log\log d)^{2\lambda}}{2\pi e^4\sqrt{d\log d}}\right)(1+o(1)),
\end{align}
where in the second line we have used the asymptotic relation 
$$d(n-2m-1)^2=n^2\log d-4\lambda n^2\log\log\log d+o(n^2).$$

\ind Let us calculate the variance of $X$. Let   $$r^\prime=\mathbb{P}(A_j|A_i)-r,$$ this being the same for any $i\neq j$. We have
\begin{align}
\label{MajorBoot:varequal}
\mbox{Var}(X)&=\sum_{i,j\in [n]\setminus[m]}(\mathbb{P}(A_j | A_i)-\mathbb{P}(A_j))\mathbb{P}(A_i)\nonumber \\
                              &=(1-r)r(n-m)+r^\prime r(n-m)(n-m-1),
\end{align}
where  the first term in (\ref{MajorBoot:varequal})is the sum over $i=j$ and the second term is the sum over $i\neq j$. Let $B_{ij}$ and $\overline{B}_{ij}$ be the events that $ij$ is, or is not, an edge in $G$ respectively. 
Note that \begin{align}\label{condB}
\mathbb{P}(A_j|B_{ij})=\mathbb{P}(B(m,p)\geq B(n-m-2,p)+1)
\end{align}
 and
\begin{align}\label{condnotB}
\mathbb{P}\left(A_j|\overline{B}_{ij}\right)=\mathbb{P}(B(m,p)\geq B(n-m-2,p)),
\end{align}
where  $B(m,p) $ and $ B(n-m-2,p) $ are independent random variables. Note that $\mathbb{P}(A_j|B_{ij})\leq \mathbb{P}\left(A_j|\overline{B}_{ij}\right)  $, hence
we may bound $r^\prime$ by
\begin{align*}
r^\prime=\mathbb{P}(A_j|A_i){-}\mathbb{P}(A_j)& = \mathbb{P}(A_j|B_{ij})\mathbb{P}(B_{ij}|A_i){+}\mathbb{P}\left(A_j|\overline{B}_{ij}\right)\mathbb{P}\left(\overline{B}_{ij}|A_i\right){-}\mathbb{P}\left(A_j\right) \\
                                   & \leq  \mathbb{P}\left(A_j|\overline{B}_{ij}\right)-\mathbb{P}\left(A_j\right) \\
                                   & =  \mathbb{P}\left(A_j|\overline{B}_{ij}\right)(1-(1-p))\\
																	 & =  p\left(\mathbb{P}\left(A_j|\overline{B}_{ij}\right)\right)\\
                                   & =  p\mathbb{P}\left(B(m,p) = B(n-m-2,p)\right),
\end{align*}
where the last equality follows from (\ref{condB}) and (\ref{condnotB}). 

\ind As $p(\frac{n}{2}-m-1)=\omega(n)\sqrt{p(1-p)n}$ (which is true for $ p\gg \frac{1}{n} $) we get from Proposition~\ref{MajorBoot:twobinomdiffupper} applied with $(N,S,T)=(\frac{n}{2}-1,\frac{n}{2}-m-1,0)$, that $r^\prime$ is at most
\begin{align*}
&\lefteqn{\frac{p(\frac{n}{2}-m-1)}{2\pi(1-p)(\frac{n}{2}-1)}\exp\left(-\frac{p(\frac{n}{2}-m-1)^2}{(1-p)(\frac{n}{2}-1)}+o(1)\right)} \\
&+\frac{3}{\pi(\frac{n}{2}-m-1)}\exp\left(-\frac{9p(\frac{n}{2}-m-1)^2}{8(1-p)(\frac{n}{2}-1)}\right)\\
\lefteqn{<\frac{p\sqrt{\log d}}{2\pi(1-p)\sqrt{d}}\exp\left(-\frac{\log d}{2}+2\lambda\log\log\log d+o(1)\right)}\\
&+\frac{6\sqrt{d}}{\pi n\sqrt{\log d}}\exp\left(-\frac{9\log d}{16}+\frac{9\lambda\log\log\log d}{4}+o(1)\right).
\end{align*}

\ind The second term is much smaller than the first term, and so (for $ n $ large enough)
\begin{align}
\label{MajorBoot:rprimebound}
%r^\prime<\left(\frac{p\sqrt{\log d}(\log\log d)^{2\lambda}}{2\pi(1-p)d}+
%\frac{6(\log\log d)^{\frac{9\lambda}{4}}}{\pi d^{\frac{1}{16}}n\sqrt{\log %d}}\right)(1+o(1)).
r^\prime<\left(\frac{\sqrt{\log d}(\log\log d)^{2\lambda}}{\pi n}\right).
\end{align}

\ind We are now able to bound the probability that $X$ is small. From (\ref{MajorBoot:varequal}) and Chebyshev's inequality we get
\begin{align*}
\mathbb{P}\left(X\leq \frac{(n-m)r}{2}\right) & \leq  \mathbb{P}\left(|X-(n-m)r|\geq \frac{(n-m)r}{2}\right) \\
                                                    & \leq  \frac{4\mbox{Var}(X)}{((n-m)r)^2} \\
                                                    & = \frac{4((1-r)+(n-m-1)r^\prime)}{(n-m)r}  \\
                                                    &< \frac{4r^\prime}{r}+o(1).
\end{align*}

\ind From (\ref{MajorBoot:rprimebound}) and (\ref{MajorBoot:rbound}) this is at most
\begin{align*}
\left(\frac{2e^4\log d\sqrt{d}}{n}\right)(1+o(1))+o(1)=o(1),
\end{align*}
and so we have, with high probability, that $|I_1\setminus I_0|$ is at least $\frac{(n-m)r}{2}$. By using (\ref{MajorBoot:rbound}) we get that for large $n$, 
$$\frac{(n-m)r}{2}\geq \frac{n(\log\log d)^{2\lambda}}{e^8\sqrt{d\log d}},$$
which completes the proof.
\end{proof}

\ind We now show that $G(n,p)$ contains no large closed sets by a simple edge set comparison.

\begin{lem}
\label{MajorBoot:nolargenasty}
Suppose that for some fixed $\epsilon>0$ we have   $$p(1-p)n\geq (1+\epsilon)\log n,$$ for $n$  large enough. Then, with high probability, $G(n,p)$ contains no closed set of size greater than $\frac{n}{2}+\frac{7n}{2\sqrt{d}}.$
\end{lem}

\begin{proof}
Let us write $s$ for the size of the set $S$ i.e., $ s=\mid S \mid $. In order for the set $S$ to be closed, each vertex $v\in [n]\setminus S$ has to have the majority of its neighbours outside $S$. In other words, we must have $|\Gamma(v)\cap([n]\setminus S)|>|\Gamma(v)\cap S|$. Summing over the vertices in $[n]\setminus S$, we have that the number of edges from $S$ to $[n]\setminus S$ must be fewer than twice the number of edges in $[n]\setminus S$. 

If $\frac{n}{2}+\frac{7n}{2\sqrt{d}}<s<\frac{4n}{5}$, then $p(2s-n)\geq7\sqrt{p(1-p)n}$, and so
$$ps(n-s)-3(n-s)\sqrt{p(1-p)s}>2p\binom{n-s}{2}+4(n-s)\sqrt{p(1-p)(n-s)}.$$

\ind By Proposition~\ref{MajorBoot:edgesinlarge} every set of size $n-s$ has at most   $$p\binom{n-s}{2}+2(n-s)\sqrt{p(1-p)(n-s)}$$ edges with probability at least $1-\frac{1}{4^{n-s}}$, and by Proposition~\ref{MajorBoot:edgesbetweenlarge} every set $S$ of size $s$ has at least   $$ps(n-s)-3(n-s)\sqrt{p(1-p)s}$$ edges between it and its complement with probability at least $1-\frac{1}{4^{n-s}}$. Therefore, with high probability, every set $S$ of size   $$\frac{n}{2}+\frac{7n}{2\sqrt{d}}<s<\frac{4n}{5}$$ is not closed. 

\ind If $s\geq\frac{4n}{5}$ and $p(1-p)n\geq 4\log n$, then we know from Proposition~\ref{MajorBoot:edgesbetweensmall} that with probability at least $1-n^{-\frac{n-s}{120}}$ there does not exist a closed set of size $s$ in $G(n,p)$. The result follows as $\sum_{i\geq 1}n^{-\frac{i}{120}}=o(1)$.

\ind If $n-n^{\frac{27}{28}}\geq s\geq\frac{4n}{5}$ and $5\log n \geq p(1-p)n\geq (1+\epsilon)\log n$, then we know from Corollary~\ref{MajorBoot:smallpmediumset} that with probability at least $1-n^{-\frac{n-s}{120}}$ there does not exist a closed set of size $s$ in $G(n,p)$.

\ind If $s\geq n-n^{\frac{27}{28}}$ and $5\log n \geq p(1-p)n \geq (1+\epsilon)\log n$, then we know from  Proposition~\ref{MajorBoot:edgesinsmall} that with probability at least $1-n^{-\frac{n-s}{120}}$ every set $S^{C}:=[n]\setminus S$ of size $n-s$ has at most $2(n-s)$ edges, and so has a vertex $v_{S^{C}}$ of degree at most $4$. By  Proposition~\ref{MajorBoot:mindegree} we have that, with high probability, the minimum degree of $G(n,p)$ is at least $9$, and so $v_{S^{C}}$ will become infected if all of $S$ is infected, and so $S$ is not closed.
%$$\frac{7ps(n-s)}{8}>2(1+\frac{n}{2(n-s)})\binom{n-s}{2},$$

%and so by Propositions~\ref{MajorBoot:edgesinsmall} and \ref{MajorBoot:edgesbetweensmall} we have with probability at least $1-2n^{-2}$ that every set of size $\frac{4n}{5}\leq s<n-1$ is not nasty.

%Proposition~\ref{MajorBoot:edgesbetweensmall} applied to sets of size $n-1$ tells us that every vertex has degree at least $\frac{7p(n-1)}{8}$ and so no set of size $n-1$ is nasty.

\end{proof}

\ind Lastly, we turn to bounding the expected number of medium sized closed sets $I_0$ is contained in. We shall therefore want a bound on the probability that a set $S$ of size at least $s$ in a particular range of $s$ is closed. To do this we shall pick a test set $T$ of a suitable size and bound the probability that none of the vertices in $T$ are infected by $S$.

\begin{lem}
\label{MajorBoot:nastyprob}
Fix $\epsilon>0$ and define  
$$s=\left\lfloor \frac{n}{2}-\frac{n\sqrt{\log d}}{2\sqrt{d}}+\frac{n(\log\log d)^{1+\epsilon}}{\sqrt{d\log d}}\right\rfloor.$$ 
\ind Take any set of vertices $S$ in $G(n,p)$ of size $s\leq|S|<\frac{2n}{3}$. Then for $ n $ large enough,
$$\mathbb{P}(S\mbox{ is closed})\leq\exp\left(-\frac{n(\log d)^{(\log\log d)^{\epsilon}-2}}{e^7\sqrt{d}}\right).$$
\end{lem}

\begin{proof}
Let $ S $ be a set of vertices such that $ s \leq \mid S \mid \leq \frac{2n}{3}$. Consider a set $ T\subset V(G)\setminus S $  of size   $t=\left\lfloor\frac{n}{(\log d)^2}\right\rfloor$. We shall condition on the edge set of $T$ as once we have done so the events $F_v$, that $v$ is not infected by $S$ for each vertex $v\in T$, are independent. 

\ind Denote by $E=E(T)$ the edge set of $T$, and set $d_E(v)$ to be the degree of vertex $v\in T$, when $T$ has edge set $E$. We have that
$$\mathbb{P}(F_v|E)=\mathbb{P}(|\Gamma(v)\cap S|<d_E(v)+|\Gamma(v)\cap([n]\setminus(S\cup T))|).$$

\ind Therefore, 
 \begin{align*}
\mathbb{P}(S \mbox{ is closed})
&\leq\sum_{E}\mathbb{P}(E)\prod_{v\in T}\mathbb{P}(F_v|E)\\
&=\sum_{E}\mathbb{P}(E)\prod_{v\in T}\mathbb{P}(B(|S|,p)<B(n-|S|-t,p)+d_E(v)),
\end{align*}
where $\mathbb{P}(E)$ is the probability of a particular edge set $E\subset \{0,1\}^{\binom{t}{2}}$  and is equal to $p^{|E|}(1-p)^{\binom{t}{2}-|E|}$.

\ind The function $f_{|S|}(x)=\mathbb{P}(B(|S|,p)<B(n-|S|-t,p)+x)$ (for independent binomial random variables $B(|S|,p)$ and $ B(n-|S|-t,p) $) is decreasing in $|S|$, so we have $f_s(x)\geq f_{|S|}(x)$. Let us supress the dependency on $s$ by writing $f(x)$ instead of $f_s(x)$. We have
\begin{align}
\label{MajorBoot:testsetedgesum}
\mathbb{P}(S \mbox{ is closed})\leq\sum_{E}\mathbb{P}(E)\prod_{v\in T}f(d_E(v)).
\end{align}

\ind The rest of the proof shall be spent bounding (\ref{MajorBoot:testsetedgesum}). The degree of vertices in $T$ is heavily concentrated around $pt$, and we shall expand $f$ around $pt$ to show that (\ref{MajorBoot:testsetedgesum}) is not much larger than $f(pt)^t$.

\ind We have by Corollary~\ref{MajorBoot:logconcor} that $f$ is log-concave, and so for any $x$ and $y$ with $f(y)\neq 0$,
$$f(x)\leq f(y)\left(\frac{f(y+1)}{f(y)}\right)^{x-y}.$$

\ind Setting $y=\lceil pt\rceil\in\mathbb{N}$ we get

\begin{align*}
\mathbb{P}(S \mbox{ is closed})&\leq\sum_{E}\mathbb{P}(E)\prod_{v\in T}f(y)\left(\frac{f(y+1)}{f(y)}\right)^{d_E(v)-y} \\
       &=\sum_{E}\mathbb{P}(E)f(y)^t\left(\frac{f(y+1)}{f(y)}\right)^{2|E|-ty}.
\end{align*}

\ind There is no dependence on $E$ other than its size, and so
\begin{align}
\label{MajorBoot:nosumtrick}
\mathbb{P}(S\mbox{ is closed})&\leq \sum_{i=0}^{\binom{t}{2}}\binom{\binom{t}{2}}{i}p^i(1-p)^{\binom{t}{2}-i}f(y)^t\left(\frac{f(y+1)}{f(y)}\right)^{2i-ty}\nonumber \\
                          &=\left(1-p+p\left(\frac{f(y+1)}{f(y)}\right)^2\right)^{\binom{t}{2}}\left(\frac{f(y)}{f(y+1)}\right)^{ty}f(y)^t.
\end{align}

\ind Setting $\frac{f(y+1)}{f(y)}=1+a,$ we bound (\ref{MajorBoot:nosumtrick}) using the inequalities $1+w\leq e^w$ and $(1+x)^{-1}\leq 1-x+x^2$ for $x\geq 0$ to get
\begin{align}
\label{MajorBoot:probbounda}
\mathbb{P}(S\mbox{ is closed})&\leq\left(1+2ap+a^2p\right)^{\frac{t^2}{2}}\left(\frac{1}{1+a}\right)^{pt^2}f(y)^t \nonumber \\
&\leq\exp\left((2ap+a^2p)\frac{t^2}{2}+(a^2-a)pt^2\right)f(y)^t\nonumber\\
                          &=\exp\left(\frac{3pa^2t^2}{2}\right)f(y)^t.      
\end{align}

\ind We have that 
$$f(y+1)=f(y)+\mathbb{P}\left(B(s,p)=B(n-s-t,p)+y\right).$$ 

\ind Let us write $z=\mathbb{P}\left(B(s,p)=B(n-s-t,p)+y\right)$ to ease up the notation. Thus, $f(y+1)=f(y)+z$. By Proposition~\ref{MajorBoot:twobinomdiffupper} applied with $N=\frac{n-t+T}{2}$, $S=\frac{n-2s-t+T}{2}$ and $T=\frac{\lceil pt\rceil}{p}$ and noting that $0\leq T-t<p^{-1}$, we have
\begin{align}
\label{MajorBoot:zbound}
z&<\frac{n-2s+\frac{1}{p}}{2\pi(1-p)n}\exp\left(-\frac{2p(\frac{n}{2}-s)^2}{(1-p)(n-t)}+o(1)\right)\nonumber\\
&{}+\frac{6}{\pi(n-2s)}\exp\left(-\frac{9p(n-2s)^2}{16(1-p)(n+\frac{1}{p})}\right)\nonumber\\
&=\frac{\sqrt{\log d}}{2\pi(1-p)\sqrt{d}}\exp\left((-\frac{\log d}{2}+2(\log\log d)^{1+\epsilon})(1+\frac{t}{n})+o(1)\right)\nonumber\\
&{}+\frac{6\sqrt{d}}{\pi n\sqrt{\log d}}\exp\left(-\frac{9\log d}{16}+\frac{9(\log\log d)^{1+\epsilon}}{4}+o(1)\right).
\end{align}

\ind The second term in (\ref{MajorBoot:zbound}) is much smaller than the first so as $6<2\pi$ and $t\log d=o(n)$ we get (for $ n $ large enough)
$$z<\frac{\sqrt{\log d}(\log d)^{2(\log\log d)^{\epsilon}}}{6(1-p)d}.$$ 

\ind We can rewrite $f(y)$ as
$$f(y)=1-\mathbb{P}\left(B(s,p)+B(n-s-t,(1-p))\geq n-s-t+y\right).$$

\ind We have for $ p\gg \frac{1}{n} $ the asymptotic relation   $$(p(n-2s)+1)(t+2s-n)=o(n\sqrt{np(1-p)}),$$ and so using Proposition~\ref{MajorBoot:twobinomlower} with $(N,S,h)=(\frac{n-t}{2},\frac{n-2s-t}{2},p(n-2s)+y-pt)$ we obtain (for $ n $ large enough) that
\begin{align}
\label{MajorBoot:f(y)upperbound}
f(y)&<1-\frac{\sqrt{p(1-p)(n-t)}}{2\pi(p(n-2s)+1)}\exp\left(-\frac{(p(n-2s)+1)^2}{2p(1-p)(n-t)}-4-o(1)\right)\nonumber\\
&<1-\frac{(\log d)^{2(\log\log)^{\epsilon}}}{e^6\sqrt{d\log d}}\nonumber\\
&<\exp\left(-\frac{(\log d)^{(\log\log d)^{\epsilon}}}{e^6\sqrt{d}}\right),
\end{align}
the second inequality follows from the same reasoning used in (\ref{MajorBoot:zbound}) and that $e^6>2\pi e^4$.

\ind We can also apply Proposition~\ref{MajorBoot:twobinomupper} to get a lower bound on $f(y)$ (for $ n $ large enough) of
$$f(y)>1-\frac{\sqrt{p(1-p)(n-t)}}{p(n-2s)}\exp\left(-\frac{p(n-2s)^2}{2(1-p)(n-t)}+3+o(1)\right)>\frac{1}{2},$$
here the bound on $1-f(y)$ is actually $o(1)$, being within a constant factor of the bound in (\ref{MajorBoot:f(y)upperbound}).   

\ind We are now able to get a good upper bound on $a$,
$$a=\frac{z}{f(y)}<\frac{\sqrt{\log d}(\log d)^{2(\log\log d)^{\epsilon}}}{3(1-p)d}.$$

\ind Substituting these bounds into (\ref{MajorBoot:probbounda}) we get
(for $ n $ large enough)
\begin{align*}
\mathbb{P}(S\mbox{ is closed})<\exp\left(\frac{p(\log d)^{4(\log\log d)^{\epsilon}}n}{6(1-p)^2d^2\log d}-\frac{(\log d)^{(\log\log d)^{\epsilon}}}{e^6\sqrt{d}}\right)^t.
\end{align*}

\ind The second term in the exponential is much larger than the first term, and so (for $ n $ large enough)
\begin{align*}
\mathbb{P}(S\mbox{ is closed})&<\exp\left(-\frac{(\log d)^{(\log\log d)^{\epsilon}}}{2e^6\sqrt{d}}\right)^t\nonumber\\
&<\exp\left(-\frac{n(\log d)^{(\log\log d)^{\epsilon}-2}}{e^7\sqrt{d}}\right),
\end{align*}
as $t>\frac{2n}{e(\log d)^2}$.
\end{proof}

\ind We shall now bound the expected number of closed sets in this medium sized range that contain $I_0$, this is also a bound on the probability that $I_0$ is contained in such a medium sized closed set.  

\begin{prop}
\label{MajorBoot:nastyexp}
\ind Assume that 
$$m=\frac{n}{2}-\frac{n\sqrt{\log d}}{2\sqrt{d}}+\frac{n\lambda \log\log\log d}{\sqrt{d\log d}}+o\left(n\frac{\log \log \log d}{\sqrt{d\log d}}\right),$$
and choose some $ \epsilon>0 $.
Then the expected number of closed sets in $G(n,p)$ of size between
$$\frac{n}{2}-\frac{n\sqrt{\log d}}{2\sqrt{d}}+\frac{n(\log\log d)^{1+\epsilon}}{\sqrt{d\log d}}\quad \text{and}\quad\frac{n}{2}+\frac{4n}{\sqrt{d}}$$
 that contain $ I_0 $ is $o(1)$.
\end{prop}

\begin{proof}
Let $S$ be a set of size $s$ in our range, $s$ can have at most $\left\lfloor\frac{n\sqrt{\log d}}{\sqrt{d}}\right \rfloor$ different values. For each possible value of $s$ and $ n $ large enough (using Stirling's formula) there are at most 
$$\binom{n-m}{s-m}<\binom{n}{\lfloor{\frac{n\sqrt{\log d}}{\sqrt{d}}\rfloor}}<\left(\frac{e\sqrt{d}}{\sqrt{\log d}}\right)^{\frac{n\sqrt{\log d}}{\sqrt{d}}}<\exp\left(\frac{n(\log d)^{\frac{3}{2}}}{\sqrt{d}}\right)$$
possible closed sets that can contain $ I_0 $. By Lemma~\ref{MajorBoot:nastyprob} the expected number of closed sets is (for $ n $ large enough) less than
$$\frac{n\sqrt{\log d}}{\sqrt{d}}\exp\left(\frac{n(\log d)^{\frac{3}{2}}}{\sqrt{d}}-\frac{n(\log d)^{(\log\log d)^{\epsilon}-2}}{e^7\sqrt{d}}\right),$$
and this is $o(1)$ as $(\log\log d)^\epsilon$ is unbounded.
\end{proof}

\begin{cor}\label{uppercor} 
Fix some number $\epsilon>0$. Assume that for $ n $ large enough,
\begin{equation*}
(1+\epsilon)\log n\leq p(1-p)n. 
\end{equation*}
If the initially infected set $I_{0}$ has size 
\begin{equation*}
m=\frac{n}{2}-\frac{n}{2}\sqrt{\frac{\log d}{d}}+\lambda n\frac{\log \log \log d}{\sqrt{d\log d}}+o\left(n\frac{\log \log \log d}{\sqrt{d\log d}}\right),
\end{equation*}
then
for $\lambda>\frac{1}{2}$, with high probability, the $MB(n,p\ ;m)$ process percolates.
\end{cor}

\begin{proof}
We have from Lemma~\ref{MajorBoot:upbstep1} that, with high probability, $I_0=[m]$ is contained in no closed set of size less than
$$\frac{n}{2}-\frac{n\sqrt{\log d}}{2\sqrt{d}}+\frac{n(\log\log d)^{2\lambda}}{e^8\sqrt{d\log d}}.$$

Using the Markov inequality it follows from Proposition~\ref{MajorBoot:nastyexp} applied to $\epsilon=\lambda-\frac{1}{2}$ that, with high probability, $I_0$ is contained in no closed set of size between
$$\frac{n}{2}-\frac{n\sqrt{\log d}}{2\sqrt{d}}+\frac{n(\log\log d)^{\lambda+\frac{1}{2}}}{\sqrt{d\log d}}\quad \text{and} \quad\frac{n}{2}+\frac{4n}{\sqrt{d}}.$$

\ind We have from Lemma~\ref{MajorBoot:nolargenasty} that, with high probability, $I_0$ is contained in no closed set of size greater than
$$\frac{n}{2}+\frac{7n}{2\sqrt{d}},$$
and so, for $\lambda>\frac{1}{2}$, with high probability, $I_0$ is not contained in any closed set in $G(n,p)$ and hence percolates.
\end{proof}

\section{Lower Bound}\label{lower}

\ind In this section we shall show the lower bound of 
Theorem \ref{MajorBoot:mainthm}. We show the following result.
\begin{lem}\label{lowerlem}
Fix some number $\epsilon>0$. Assume that for $ n $ large enough,
\begin{equation*}
(1+\epsilon)\log n\leq p(1-p)n. 
\end{equation*}
If the initially infected set $I_{0}$ has size 
\begin{equation*}
m=\frac{n}{2}-\frac{n}{2}\sqrt{\frac{\log d}{d}}+\lambda n\frac{\log \log \log d}{\sqrt{d\log d}}+o\left(n\frac{\log \log \log d}{\sqrt{d\log d}}\right),
\end{equation*}
then, for $\lambda<0$, with high probability, the $MB(n,p\ ;m)$ process does not percolate. 
\end{lem}
\begin{rem}
Note that Lemma \ref{lowerlem} and Corollary \ref{uppercor} prove Theorem \ref{MajorBoot:mainthm}.
\end{rem}

In fact to prove Lemma \ref{lowerlem}, as might be expected, we shall show that, with high probability, the $MB(n,p\ ;m)$ process terminates with $I$ (the set of eventual infected vertices) only slightly larger than $\mid I_0 \mid=m$. We shall do this by bounding the expected number of sets of some size that could be the first vertices to be infected. 

\begin{proof}
 We say that a set of vertices $T$ percolates if all of its vertices will be infected eventually. For $T\subset I\setminus I_0$ we can order the vertices of $I_0\cup T$ by the time they get infected. That is, take any order of $T$ such that a vertex from $I_{j}$ is infected before any vertex from $I_{j'}$ if $j<j'$. Notice that for each $v\in T$ the majority of its neighbours (in the whole graph) are in the set of its predecessors in this order. Our strategy will be to show that if $\lambda < 0$ then, with high probability, there is no percolating set $T$ of a particular size and thus the $MB(n,p\ ;m)$ process does not percolate.

\ind Set $t=\left|T\right|$, and denote by $E=E(T)$ the edge set of $T$. Write $d_E(i)$ for the degree within $T$ of a vertex $i\in T$. %Condition on the edge configuration $E$.
 We want to bound the probability that $T$ percolates. To do so, we modify the infection rule within $T$ so that the vertices inside $T$ consider their neighbours in $T$ to be already infected, regardless of their real state at any particular time step. The latter assumption only increases the probability and, more importantly, makes the events for vertices in $T$ to be infected independent. This is because these events now only depend on how many edges each vertex has to $I_{0}$ and $V(G)/\left(I_{0}\cup T\right)$. Conditioning on $ E $, and then taking the expectation give
\begin{equation}
\label{MajorBoot:lb}
\mathbb{P}(T\mbox{ percolates})\leq\sum_{E}\mathbb{P}(E)\prod_{i=1}^t\mathbb{P}(B(m,p)+d_E(i)\geq B(n-m-t,p))
\end{equation} 
(for independent random variables $ B(m,p) $ and $ B(n-m-t,p) $).
\ind Denote $g(x)= \mathbb{P}\left(B(m,p)+x\geq B(n-m-t,p)\right)$. Due to the $\log$-concavity of $g$ (Corollary~\ref{MajorBoot:logconcor}) we have for integers $x,y,$ that
$$g(x)\leq g(y)\left(\frac{g(y+1)}{g(y)}\right)^{x-y}.$$
Using the latter inequality with $x=d_E(i)$ and $y=\lceil pt\rceil$, we can bound (\ref{MajorBoot:lb}) by
\begin{align*}
&\sum_{E}\mathbb{P}(E)\prod_{i=1}^{t}g(y)\left(\frac{g(y+1)}{g(y)}\right)^{d_{E}(i)-y}\\
&=\sum_{E}\mathbb{P}(E)g(y)^t\left(\frac{g(y+1)}{g(y)}\right)^{2|E|-ty}\\
&=\sum_{j=0}^{\binom{t}{2}}\binom{\binom{t}{2}}{j}p^{j}(1-p)^{\binom{t}{2}-j}g(y)^t\left(\frac{g(y+1)}{g(y)}\right)^{2j-ty}\\
&=\left(1-p+p\left(\frac{g(y+1)}{g(y)}\right)^{2}\right)^{\binom{t}{2}}\left(\frac{g(y)}{g(y+1)}\right)^{pt^2}g(y)^t.
\end{align*}

\ind Substituting $\frac{g(y+1)}{g(y)}=1+a$ and the elementary inequality $1/(1+a)\leq 1-a+a^2$, we bound the latter expression by 
\begin{align}
\label{MajorBoot:lbnicea}
&\left(1-p+p(1+a)^{2}\right)^{\binom{t}{2}}\left(1-a+a^2\right)^{ty}g(y)^t\nonumber\\
&\leq \exp \left((2ap+a^{2}p)\frac{t^2}{2}+(a^2-a)pt^2\right)g(y)^t\nonumber\\
&= \left(\exp \left(\frac{3pa^{2}t}{2}\right)g(y)\right)^t.
\end{align}

\ind We have by definition that $g(y)$ is equal to

$$g(y)=\mathbb{P}\left(X_1+X_2\geq \mu_1+\mu_2+pn-2pm-pt-\lceil pt\rceil\right),$$
where $X_1=B(m,p)$ with mean $\mu_1$ and $X_2=B(n-m-t,(1-p))$ with mean $\mu_2$. Setting $t=\left\lfloor n(\log \log d)^{\lambda}/\sqrt{d\log d}\right\rfloor$ and using Proposition~\ref{MajorBoot:twobinomupper} with $N=\frac{n-t}{2}$, $S=\frac{n-2m-t}{2}$ and $h=p(n-2m-t)-y$  to bound $g(y)$, we obtain (for $ n $ large enough)
%\begin{eqnarray*}
%&&\mathbb{P}\left(T\mbox{ percolates}\right)\leq \left((1+o(1))\frac{\sqrt{np(1-p)}}{\sqrt{2\pi}(n-2m-2t)p}\exp\left(-\frac{(n-2m-2t)^{2}p^{2}}{2p(1-p)n}+o(1)\right)\right)^{t}\\
%&\leq&\left((1+o(1))\frac{\sqrt{np(1-p)}}{\sqrt{2\pi}np\sqrt{\frac{\log d}{d}}}\exp\left(-\frac{(n\sqrt{\frac{\log d}{d}}-2\lambda n\frac{\log \log \log d}{\sqrt{d\log d}}-2n\frac{(\log \log d)^{\lambda}}{\sqrt{d\log d}})^{2}p^{2}}{2p(1-p)n}\right)\right)^{t}\\
%&\leq&\left(\frac{\frac{1}{\sqrt{2\pi}}+o(1)}{\sqrt{\log d}}\exp\left(\frac{-n^{2}p^{2}\frac{\log d}{d}+4(1+o(1))\lambda n^{2}p^{2}\frac{\log \log \log d}{d}}{2p(1-p)n}\right)\right)^{t}\\
%&\leq& \left(\frac{\frac{1}{\sqrt{2\pi}}+o(1)}{\sqrt{\log d}}\exp\left(-\frac{\log d}{2}+2(1+o(1))\lambda \log \log \log d \right)\right)^{t}.
%\end{eqnarray*}
\begin{align*}
g(y)&<\frac{\sqrt{p(1-p)(n-t)}}{pn-2pm-2pt-1}\exp\left(-\frac{(pn-2pm-2pt-1)^2}{2p(1-p)(n-t)}+3+o(1)\right)\\
&<\frac{e^3}{\sqrt{\log d}}\exp\left(-\frac{\log d}{2}+2\lambda\log\log\log d+O((\log\log d)^{\lambda})\right)\\
&<\left(\frac{e^4(\log\log d)^{2\lambda}}{\sqrt{d\log d}}\right),
\end{align*}
when $\lambda<0$.\\
\ind We can also bound $g(y)$ from below by Proposition~\ref{MajorBoot:twobinomlower}
\begin{align}
\label{MajorBoot:lbg(y)}
g(y)&>\frac{\sqrt{p(1-p)(n-t)}}{2\pi(pn-2pm-2pt)}\exp\left(-\frac{(pn-2pm-2pt)^2}{2p(1-p)(n-t)}-4-o(1)\right)\nonumber\\
&>\frac{1}{2\pi e^4\sqrt{\log d}}\exp\left(-\frac{\log d}{2}+2\lambda\log\log\log d+O((\log\log d)^{\lambda})\right)\nonumber\\
&>\left(\frac{(\log\log d)^{2\lambda}}{e^6\sqrt{d\log d}}\right),
\end{align} 
when $\lambda<0$ (and $ n $ is large enough).\\ 
\ind By definition of $g$ we have that   $$g(y+1)=g(y)+\mathbb{P}(B(m,p)+y+1=B(n-m-t,p)).$$ Let us write $z=\mathbb{P}(B(m,p)+y+1=B(n-m-t,p))$ for convenience. We shall now obtain an upper bound for $z$. Using Proposition~\ref{MajorBoot:twobinomdiffupper} with $T=-\frac{y+1}{p}$, $N=\frac{n-t+T}{2}$ and $S=N-m$, we obtain
\begin{align*}
z &<\frac{\frac{n}{2}-m-2t}{2\pi(1-p)(\frac{n}{2}-2t-\frac{2}{p})}\exp\left(-\frac{2p(\frac{n}{2}-m-2t-\frac{2}{p})^2}{(1-p)(n-t)}+o(1)\right)\\
&{}+\frac{3}{\pi p(\frac{n}{2}-m-2t-\frac{2}{p})}\exp\left(-\frac{9p(\frac{n}{2}-m-2t-\frac{2}{p})^2}{8(1-p)(\frac{n}{2}-2t-\frac{1}{p})}\right)\\
&<\frac{\sqrt{\log d}}{2\pi(1-p)\sqrt{d}}\exp\left(-\frac{\log d}{2}+2\lambda\log\log\log d+o(1)\right)\\
&{}+\frac{6\sqrt{d}}{\pi pn\sqrt{\log d}}\exp\left(-\frac{9\log d}{16}+\frac{9\lambda\log\log\log d}{4}+o(1)\right).
\end{align*}

\ind The first term is much larger than the second, and so we obtain (for $ n $ large enough) the inequality
\begin{equation}
\label{MajorBoot:lbz}
z<\frac{\sqrt{\log d}(\log\log d)^{2\lambda}}{\pi(1-p)d}.
\end{equation}

\ind We have that $a=\frac{z}{g(y)},$ and so from (\ref{MajorBoot:lbg(y)}) and (\ref{MajorBoot:lbz}) (for $ n $ large enough)
\begin{align*}
a<\frac{e^6\log d}{\pi(1-p)\sqrt{d}}<\frac{e^5\log d}{(1-p)\sqrt{d}}.
\end{align*}

\ind We can now bound the expression in (\ref{MajorBoot:lbnicea}) (for $ n $ large enough) by
\begin{align*}
\mathbb{P}(T\mbox{ percolates})&<\left(\exp\left(\frac{3pe^{10}(\log d)^2n(\log\log d)^{\lambda}}{2(1-p)^2d\sqrt{d\log d}}\right)\frac{e^4(\log\log d)^{2\lambda}}{\sqrt{d\log d}}\right)^t\\
&<\left(\frac{e^5(\log\log d)^{2\lambda}}{\sqrt{d\log d}}\right)^t
\end{align*}
(where the second inequality follows since the exponent in the exponential is $ o(1) $).

\ind The expected number of sets of size $t$ that percolates is (for $ n $ large enough)
\begin{align*}
\binom{n-m}{t}\mathbb{P}(T\mbox{ percolates})&<\binom{n}{t}\left(\frac{e^5(\log\log d)^{2\lambda}}{\sqrt{d\log d}}\right)^t\\
&< \left(\frac{e^6n(\log\log d)^{2\lambda}}{t\sqrt{d\log d}}\right)^t,
\end{align*}
because $\binom{n}{t}\leq\left(\frac{en}{t}\right)^t$. We chose $t=\left\lfloor\frac{n(\log\log d)^\lambda}{\sqrt{d\log d}}\right\rfloor,$ and so the expected number of sets of size $t$ that percolates is bounded above by
\begin{equation*}
(e^6(\log\log d)^{\lambda})^t=o(1).
\end{equation*}

%Using the latter bound we have 

%\begin{eqnarray*}
%&&\mathbb{P}\left(\exists \text{\,\, a percolating set of size\,\,} t\right)\leq \binom{n}{t}\left[(\frac{1}{\sqrt{2\pi}}+o(1))\exp\left(2(1+o(1))\lambda \log \log \log d \right)\right]^{t}\\
%&\leq& \left[\frac{en}{t}\frac{\frac{1}{\sqrt{2\pi}}+o(1)}{\sqrt{\log d}}\exp\left(-\frac{\log d}{2}+2(1+o(1))\lambda \log \log \log d\right]^{t}\\
%&\leq& \left[(3+o(1))\sqrt{d}(\log \log d)^{-\lambda}\exp\left(-\frac{\log d}{2}+2(1+o(1))\lambda \log \log \log d\right]^{t}\\
%&=& \left[(3+o(1))\exp\left((1+o(1))\lambda \log \log \log d\right]^{t}=o(1).\\
%\end{eqnarray*}

Therefore, with high probability, percolation does not occur for $\lambda<0$.
\end{proof}

\section{Inequalities}\label{ineqineq}

\ind We begin this section with some remarks on the log-concavity of the distribution function of the Binomial distribution. These results are standard, see for example \cite{keilson1971some}, but we prove them for completeness. 

\begin{prop}
\label{MajorBoot:logcon}
The sum of independent Bernoulli random variables is log-concave, that is if $X_i$ are independent Bernoulli random variables with means $p_i,$ then for any $k$ we have,
$$\mathbb{P}(\sum_{i=1}^n X_i=k-1)\mathbb{P}(\sum_{i=1}^n X_i=k+1)\leq (\mathbb{P}(\sum_{i=1}^n X_i=k))^2.$$
\end{prop}

\begin{proof}
We proceed by induction on $n,$ with the base case $n=1$ being trivial as one of the terms on the left hand side of the inequality is zero. Otherwise conditioning on $X_{n+1},$ and writing $f_{n,k}=\mathbb{P}(\sum_{i=1}^n X_i=k)$ we get, 

\begin{align*}
f_{n{+}1,k{-}1}f_{n{+}1,k{+}1} & =  (p_{n{+}1}f_{n,k{-}2}{+}(1{-}p_{n{+}1})f_{n,k{-}1})(p_{n{+}1}f_{n,k}{+}(1{-}p_{n{+}1})f_{n,k{+}1})  \\
                 & \leq  (p_{n{+}1}f_{n,k{-}1}{+}(1{-}p_{n{+}1})f_{n,k})^2  \\ & = (f_{n+1,k})^2
\end{align*} 

\ind The inequality follows as $f_{n,k{-}2}f_{n,k{+}1}\leq f_{n,k-1}f_{n,k}$ is implied by the induction hypothesis.
\end{proof}

\begin{prop}
\label{MajorBoot:logconcum}
The cumulative distribution of a discrete non-negative log-concave random variable $X$ is log-concave, that is for all $k,$
$$\mathbb{P}(X\leq k-1)\mathbb{P}(X\leq k+1)\leq (\mathbb{P}(X\leq k))^2.$$
\end{prop}

\begin{proof}
Setting $r_i=\mathbb{P}(X=i)$ we get by Proposition~\ref{MajorBoot:logcon},
$$(r_0+\ldots+r_{k-1})r_{k+1} \leq (r_1+\ldots+r_k)r_k +r_kr_0,$$
and so,
$$(r_0+\ldots+r_{k-1})(r_0+\ldots+r_{k+1})\leq (r_0+\ldots+r_k)^2.$$ 
\end{proof}

\ind When $X$ is the sum of $n$ independent Bernoulli random variables, we can rewrite $X=n-Y,$ where $Y$ is also the sum of $n$ independent Bernoulli random variables, and so Proposition~\ref{MajorBoot:logconcum} is still true if we replace $\leq$, with $<$, $>$ or $\geq$.

\begin{cor}
\label{MajorBoot:logconcor}
The cumulative distribution of the sum or difference of independent binomial random variables is log-concave. 
\end{cor}

\begin{proof}
Sums and differences of independent binomial random variables are also sums of independent Bernoulli random variables plus a constant, and so are log-concave. 
\end{proof} 

\ind A substantial part of this section is now taken up with providing tight bounds, up to a constant factor, on binomial probabilities and their sums.

\begin{prop}
\label{MajorBoot:binomlower}
Suppose $pn\geq 1$ and $k=pn+h<n$, where $h>0$. Set
$$\beta = \frac{1}{12k}+\frac{1}{12(n-k)},$$
then $\mathbb{P}(B(n,p)=k)$ is at least
$$\frac{1}{\sqrt{2\pi p(1{-}p)n}}\exp{\left({-}\frac{h^2}{2p(1{-}p)n}-\frac{h^3}{2(1{-}p)^2n^2}-\frac{h^4}{3p^3n^3}-\frac{h}{2pn}-\beta\right)}.$$
\end{prop}

\begin{proof}
This is Theorem 1.5 in \cite{bollobas2001random}, p. 12.
\end{proof}

\begin{cor}
\label{MajorBoot:binomlowercor}
Suppose $p(1-p)n=\omega(n)$ and $k=pn+h,$ where $0<h=o\left((p(1-p)n)^{\frac{2}{3}}\right),$ then
$$\mathbb{P}(B(n,p)=k)>\frac{1}{\sqrt{2\pi p(1-p)n}}\exp\left(-\frac{h^2}{2p(1-p)n}-o(1)\right).$$
\end{cor}

\begin{proof}
For $h$ in this range we have 
$$\frac{h^3}{2(1-p)^2n^2}+\frac{h^4}{3p^3n^3}+\frac{h}{2pn}=o(1).$$

\ind We also have that $k=\omega(n)$ and $n-k=\omega(n)$, and so the inequality follows from Proposition~\ref{MajorBoot:binomlower}.
\end{proof}

\begin{prop}
\label{MajorBoot:binomupper}
Suppose $pn\geq 1$ and $k\geq pn+h$, where $h(1-p)n\geq 3$. Then
$$\mathbb{P}(B(n,p)=k)<\frac{1}{\sqrt{2\pi p(1{-}p)n}}\exp\left({-}\frac{h^2}{2p(1{-}p)n}{+}\frac{h^3}{p^2n^2}{+}\frac{h}{(1{-}p)n}\right).$$
\end{prop}

\begin{proof}
This is Theorem 1.2 of \cite{bollobas2001random}, p. 10.
\end{proof}

\begin{cor}
\label{MajorBoot:binomuppercor}
Suppose $p(1-p)n=\omega(n)$ and $k\geq pn+h,$ where   $$1<h=o((p(1-p)n)^{\frac{2}{3}}),$$ then
$$\mathbb{P}(B(n,p)=k)<\frac{1}{\sqrt{2\pi p(1-p)n}}\exp\left(-\frac{h^2}{2p(1-p)n}+o(1)\right).$$
\end{cor}

\begin{proof}
For $h$ in this range we have 
$$\frac{h^3}{p^2n^2}+\frac{h}{(1-p)n}=o(1),$$
and so the inequality follows from Proposition~\ref{MajorBoot:binomupper}, which can be applied as $h(1-p)n=\omega(n)$.
\end{proof}

\begin{prop}
\label{MajorBoot:binomcumupper}
Suppose $p(1-p)n=\omega(n)$ and $0<h=o((p(1-p)n)^{\frac{2}{3}})$, then
$$\mathbb{P}(B(n,p)\geq pn+h)<\frac{\sqrt{p(1-p)n}}{\sqrt{2\pi}h}\exp\left(-\frac{h^2}{2p(1-p)n}+o(1)\right).$$
\end{prop}

\begin{proof}
This proof follows that of Theorem 1.3 in \cite{bollobas2001random}. For $m\geq pn+h,$ we have 
$$\frac{\mathbb{P}(B(n,p)=m+1)}{\mathbb{P}(B(n,p)=m)}\leq 1-\frac{h+(1-p)}{(1-p)(pn+h+1)}=\lambda.$$
Hence,
$$\mathbb{P}(B(n,p)\geq pn+h)\leq \frac{1}{1-\lambda}\mathbb{P}(B(n,p)=\lceil pn+h\rceil).$$

As $(1-\lambda)^{-1}<\frac{p(1-p)n}{h}(1+\frac{h}{pn})<\frac{p(1-p)n}{h}e^{\frac{h}{pm}},$ we get from Proposition~\ref{MajorBoot:binomupper} that
$$\mathbb{P}(B(n,p)\geq pn+h)<\frac{\sqrt{p(1{-}p)n}}{h\sqrt{2\pi}}\exp\left({-}\frac{h^2}{2p(1{-}p)n}{+}\frac{h}{p(1{-}p)n}{+}\frac{h^3}{p^2n^2}\right)$$
the last two terms in the exponent being $o(1),$ for $h=o(p(1-p)n)^{\frac{2}{3}}$.
\end{proof}

\begin{prop}
\label{MajorBoot:binomcumlower}
Suppose $p(1-p)n=\omega(n)$ and   $$(p(1-p)n)^{\frac{1}{2}}<h=o((p(1-p)n)^{\frac{2}{3}}),$$ then
$$\mathbb{P}(B(n,p)\geq pn+h)>\frac{\sqrt{p(1-p)n}}{h\sqrt{2\pi}}\exp\left(-\frac{h^2}{2p(1-p)n}-\frac{3}{2}-o(1)\right).$$
\end{prop}

\begin{proof}
Due to the unimodality of the binomial distribution, we have that the probability density function of the binomial distribution is decreasing away from its mean, and so,
$$\mathbb{P}(B(n,p)\geq pn+h)>\frac{p(1-p)n}{h}\mathbb{P}(B(n,p)=pn+h+\frac{p(1-p)n}{h}).$$

\ind We can apply Corollary~\ref{MajorBoot:binomlowercor} as $h+\frac{p(1-p)n}{h}=o((p(1-p)n)^\frac{2}{3})$, and so it follows that 
$$\mathbb{P}(B(n,p)\geq pn+h)>\frac{\sqrt{p(1-p)n}}{h\sqrt{2\pi}}\exp\left(-\frac{(h+\frac{p(1-p)n}{h})^2}{2p(1-p)n}-o(1)\right).$$

\ind This is greater than the stated bound because   $$(h+\frac{p(1-p)n}{h})^2\leq h^2+3p(1-p)n.$$
\end{proof}

\ind We shall also want a weaker but more general bound than Proposition~\ref{MajorBoot:binomcumupper} due to Bernstein in \cite{bernstein1924modification}. 

\begin{lem}
\label{MajorBoot:Bernstein}
Let $X_1,\ldots,X_n$ be independent zero-mean random variables. Suppose that $|X_i|\leq M$, then for all positive $t,$ 
$$\mathbb{P}\left(\sum_{i=1}^nX_i> t\right)\leq\exp\left(-\frac{t^2}{2\sum\mathbb{E}(X_j^2)+\frac{2}{3}Mt}\right).$$

\end{lem}

\begin{proof}
For a proof see \cite{craig1933tchebychef}. 
\end{proof}

We have in this section, so far discussed well-known deviation inequalities for standard binomial distributions. We will now proceed to present some analogous results for sums of binomial distributions with different parameters $ p $, that we have not  been able to find in the litterature.
\begin{prop}
\label{MajorBoot:twobinomlower}
Suppose that $p(1-p)N=\omega(N),$ the inequality   $$2(2p(1-p)N)^{\frac{1}{2}}<h=o((p(1-p)N)^{\frac{2}{3}})$$ holds and   $$hS=o(N((p(1-p)N)^{\frac{1}{2}})).$$ For the independent random variables; $X_1=B(N-S,p),$ with mean $\mu_1$ and variance $\sigma_1^2,$; and $X_2=B(N+S,(1-p))$ with mean $\mu_2$ and variance $\sigma_2^2$, we have
  $$\mathbb{P}(X_1+X_2\geq \mu_1+\mu_2+h)>\frac{\sqrt{2p(1{-}p)N}}{2\pi h}\exp\left(-\frac{h^2}{4p(1{-}p)N}-4-o(1)\right).$$

\end{prop}

\begin{proof}
The conditions on $S$ and $h$ imply that $S=o(N)$. Set $z$ and $l$ equal to $\frac{2p(1-p)N}{h}$ and $\left\lfloor\frac{h}{\sqrt{2p(1-p)N}}\right\rfloor$ respectively. We can bound   $$\mathbb{P}(X_1+X_2\geq \mu_1+\mu_2+h)$$ from below by summing over the disjoint regions 
\begin{equation}
\label{MajorBoot:threerelations}
\sum_{i=-l}^{l-1}\mathbb{P}\left(X_1<\mu_1{+}\frac{h}{2}{-}iz, X_2<\mu_2{+}\frac{h}{2}{+}(i{+}1)z, X_1{+}X_2\geq \mu_1{+}\mu_2{+}h\right).
\end{equation}
These regions are disjoint as if $ X_1<\mu_1{+}\frac{h}{2}{-}(i+1)z $ and $ X_2<\mu_2{+}\frac{h}{2}{+}(i{+}1)z $, then $ X_1+X_2<  \mu_1{+}\mu_2{+}h$.
For each $i$ the region specified is an
isosceles right angled triangle with axis-parallel legs of length $ z $, and so there are at least $\lfloor z\rfloor(\lfloor z\rfloor-1)/2$ pairs of integer values $x_1,x_2$, which $X_1,X_2$ can take while still satisfying all three relations in~(\ref{MajorBoot:threerelations}). We have that $h>2lz$, and so if $X_1,X_2$ satisfy all three relations in~(\ref{MajorBoot:threerelations}), then $X_1\geq\mu_1$ and $X_2\geq\mu_2$. As we are only considering the region in which $X_1,X_2$ are larger than their means we can bound the sum in (\ref{MajorBoot:threerelations}) from below by
\begin{equation}
\label{MajorBoot:twobinomlowersum}
\sum_{i=-l}^{l-1}\frac{\lfloor z\rfloor(\lfloor z\rfloor{-}1)}{2}\mathbb{P}\left(X_1=\left\lceil\mu_1+\frac{h}{2}-iz\right\rceil\right)\mathbb{P}\left(X_2=\left\lceil\mu_2+\frac{h}{2}+(i+1)z\right\rceil\right).
\end{equation}

\ind We have that $p(1-p)(N-S)=\omega(N)$ and $h+lz=o(p(1-p)(N-S))^{\frac{2}{3}},$ and so we can apply Corollary~\ref{MajorBoot:binomlowercor} to get that the quantity in (\ref{MajorBoot:twobinomlowersum}) is at least
\begin{align*}&\sum_{i=-1}^{l-1}\frac{\lfloor z\rfloor(\lfloor z\rfloor{-}1)}{4\pi \sigma_1\sigma_2}\cdot \\&\exp\left(-\frac{(\frac{h}{2}{-}iz{+}1)^2(N{+}S){+}(\frac{h}{2}{+}(i{+}1)z{+}1)^2(N{-}S)}{2p(1-p)(N^2-S^2)}{-}o(1)\right).
\end{align*}

\ind Expanding this out, and noticing $\left\lfloor z\right\rfloor=z(1+o(1))$ and   $$(N-S)(N+S)=N^2(1+o(1))$$ we get that the sum in (\ref{MajorBoot:twobinomlowersum}) is at least 
\begin{align}\label{ny}&\sum_{i=-l}^{l-1}\frac{z^2}{4\pi p(1{-}p)N}\cdot \nonumber\\&
\exp\left(-\frac{h^2N{+}2hzN{+}4i^2z^2N{+}(4i{+}2)z^2N{+}o(p(1{-}p)N^2)}{4p(1{-}p)(N^2{-}S^2)}{-}o(1)\right),
\end{align}
where the approximations for $\left\lfloor z\right\rfloor$ and $\sigma_1,\sigma_2$ have been taken care of in the $o(1)$ in the exponential term. We have that $ 4i^{2}+4i+2\leq 6 l^{2} $ and $ l^{2}z^{2}\leq 2p(1-p)N $, and so using the bounds in the statement of the proposition, the sum in (\ref{ny}) is at least
\begin{align*}&\sum_{i=-l}^{l-1}\frac{z^2}{4\pi p(1{-}p)N}
\exp\left(-\frac{h^2N{+}16p(1{-}p)N^2}{4p(1{-}p)(N^2{-}S^2)}{-}o(1)\right)\\
&\lefteqn{\frac{lz^2}{2\pi p(1-p)N}\exp\left(-\frac{h^2N}{4p(1-p)(N^2-S^2)}-4-o(1)\right)}\\
&>\frac{\sqrt{2p(1-p)N}}{2\pi h}\exp\left(-\frac{h^2}{4p(1-p)N}-4-o(1)\right).
\end{align*}

\ind The last inequality following because $l>h/(2\sqrt{2p(1-p)N})$ and $hS=o(N(p(1-p)N)^{\frac{1}{2}})$. 
\end{proof}

\begin{prop}
\label{MajorBoot:twobinomupper}
Suppose that $p(1-p)N=\omega(N)$. Furthermore assume that 
$$2(2p(1-p)N)^{\frac{1}{2}}<h=o((p(1-p)N)^{\frac{2}{3}})\quad \text{and}\quad Sh=o(N(p(1-p)N)^{\frac{1}{2}}).$$ 
\ind Then we have
$$\mathbb{P}(X_1+X_2\geq \mu_1+\mu_2+h)<\frac{\sqrt{2p(1-p)N}}{h}\exp\left(-\frac{h^2}{4p(1-p)N}+3+o(1)\right),$$
for independent random variables $X_1=B(N-S,p)$ with mean $\mu_1$ and variance $\sigma_1^2$, and $X_2=B(N+S,(1-p))$ with mean $\mu_2$ and variance $\sigma_2^2$.
\end{prop}

\begin{proof}
The conditions on $S$ and $h$ imply that $S=o(N)$. Set $z=\frac{2Np(1-p)}{h},$ and $l=\left\lfloor\frac{h^2}{4Np(1-p)}\right\rfloor$. We bound $\mathbb{P}(X_1+X_2\geq \mu_1+\mu_2+h)$ from below by covering the region where this inequality holds by 
\begin{align}
\label{MajorBoot:twobinomuppersum}
&\lefteqn{\mathbb{P}(X_1+X_2\geq \mu_1+\mu_2+h)<}\\
\label{MajorBoot:line1}
&\sum_{i+j\geq -1}^{-l\leq i,j\leq l-1}\left(\mathbb{P}\left(0\leq X_1{-}\mu_1{-}\frac{h}{2}{-}iz< z\right)\mathbb{P}\left(0\leq X_2{-}\mu_2{-}\frac{h}{2}{-}jz<z\right)\right)\\
\label{MajorBoot:line2}
&+\mathbb{P}\left(X_1\geq \mu_1+\frac{h}{2}+lz\right)\\
\label{MajorBoot:line3}
&+\mathbb{P}\left(X_2\geq \mu_2+\frac{h}{2}+lz\right).
\end{align}

\ind We shall bound these three summands separately. 
Again because $h>2lz$ we are only considering the range in which $X_1$ and $X_2$ are greater than their means.
Firstly for each $i,j$ pair there are at most $\lceil z\rceil^2$ points inside the specified region, and so the product inside the sum of (\ref{MajorBoot:line1}) is at most
$$\lceil z\rceil^2\mathbb{P}\left(X_1=\left\lceil\mu_1+\frac{h}{2}+iz\right\rceil\right)\mathbb{P}\left(X_2=\left\lceil\mu_2+\frac{h}{2}+jz\right\rceil\right).$$

\ind  We have that $p(1-p)(N\pm S)=\omega(N)$ and   $$1<h\pm lz=o(p(1-p)(N\pm S))^{\frac{2}{3}},$$ and so we can apply Corollary~\ref{MajorBoot:binomuppercor} to get that the sum in (\ref{MajorBoot:line1}) is at most

\begin{align*}
&\lefteqn{\sum_{i+j\geq -1}^{-l\leq i,j\leq l-1}\frac{\lceil z\rceil^2}{2\pi p(1-p)\sqrt{N^2-S^2}}}\\
&\cdot\exp\left(-\frac{\left(\frac{h}{2}+iz\right)^2(N+S)+\left(\frac{h}{2}+jz\right)^2(N-S)}{2p(1-p)(N^2-S^2)}+o(1)\right).
\end{align*}

\ind This is equal to
\begin{align}\label{MajorBoot:nastysum}
&\lefteqn{\frac{z^2}{2\pi p(1-p)N}\exp\left(-\frac{h^2N}{4p(1-p)(N^2-S^2)}+o(1)\right)\sum_{i+j\geq-1}^{-l\leq i,j<l}}\nonumber \\
&\exp\left(-\frac{h(i+j)zN+hzS(i-j)+z^2N(i^2+j^2)+z^2S(i^2-j^2)}{2p(1-p)(N^2-S^2)}\right).
\end{align}
\ind We can bound the above by noting that $|i-j|\leq\sqrt{2}\sqrt{i^2+j^2}$ and $|i^2-j^2|\leq i^2+j^2$. As we also have that $Np(1-p)/2<z^2l\leq Np(1-p)$, the inner sum appearing in (\ref{MajorBoot:nastysum}) is at most
\begin{align}
\label{MajorBoot:ijsum}
\sum_{i+j\geq-1}^{-l\leq i,j<l}\exp\left(-(i+j)+\sqrt{\frac{i^2+j^2}{4l}}-\frac{i^2+j^2}{4l}+o(1)\right).
\end{align}

\ind A point $(i,j)$ in the plane with integer coordinates and   $$\frac{i^2+j^2}{4l}-\sqrt{\frac{i^2+j^2}{4l}}<t,$$ also satisfies $|i-j|<\sqrt{21tl}$, as if $ |i-j|\geq\sqrt{21tl} $, then $ i^2+j^2\geq \frac{21tl}{2} $, and so $$\frac{i^2+j^2}{4l}-\sqrt{\frac{i^2+j^2}{4l}}\geq\left(\frac{21}{8}-\sqrt{\frac{21}{8}}\right)t.$$ Therefore the number of points $(i,j)$ in the plane with integer coordinates and satisfying both $\frac{i^2+j^2}{4l}-\sqrt{\frac{i^2+j^2}{4l}}<t,$ and $-1\leq i+j<t$ is at most $2(t+1)\sqrt{21lt}$. This allows us crudely bound (\ref{MajorBoot:ijsum}) by
\begin{align*}
2\sqrt{21l}\sum_{t=1}^\infty (t+1)\sqrt{t}\exp\left(-(t-1)\right).
\end{align*}

\ind The latter sum is less than $50\sqrt{l},$ and so the sum in (\ref{MajorBoot:line1}) is bounded above by
\begin{equation}
\label{MajorBoot:bigsumpart1}
\frac{50\sqrt{p(1-p)N}}{h\pi}\exp\left(-\frac{h^2}{4p(1-p)N}+o(1)\right).
\end{equation}

\ind Secondly we bound the probability (\ref{MajorBoot:line2}). As $l>\frac{h^2}{8Np(1-p)}$ we have that
$$\mathbb{P}\left(X_1\geq \mu_1+\frac{h}{2}+lz\right)<\mathbb{P}\left(X_1\geq \mu_1+\frac{3h}{4}\right).$$ 

\ind By Proposition~\ref{MajorBoot:binomcumupper} we get that the quantity in (\ref{MajorBoot:line2}) is at most
\begin{align}
\label{MajorBoot:bigsumpart2}
&\lefteqn{\frac{4\sqrt{p(1-p)(N-S)}}{3h\sqrt{2\pi}}\exp\left(-\frac{9h^2}{32p(1-p)(N-S)}+o(1)\right)}\nonumber\\
&<\frac{2}{3\sqrt{\pi}}\frac{\sqrt{2p(1-p)N}}{h}\exp\left(-\frac{h^2}{4p(1-p)N}+o(1)\right).
\end{align}

Similarly, the probability in (\ref{MajorBoot:line3}) is at most
\begin{equation}
\label{MajorBoot:bigsumpart3}
\frac{2}{3\sqrt{\pi}}\frac{\sqrt{2p(1-p)N}}{h}\exp\left(-\frac{h^2}{4p(1-p)N}+o(1)\right).
\end{equation}

As $\frac{50}{\sqrt{2}\pi}+\frac{4}{3\sqrt{\pi}}<e^3$ we get that the sum of our three bounds, (\ref{MajorBoot:bigsumpart1}), (\ref{MajorBoot:bigsumpart2}), and (\ref{MajorBoot:bigsumpart3}) is at most the stated bound.
\end{proof}

\begin{prop}
\label{MajorBoot:twobinomdiffupper}
Suppose that $p(1-p)N=\omega(N)$, that   $$\omega(N)(p(1-p)N)^{\frac{1}{2}}\leq pS=o((p(1-p)N)^{\frac{2}{3}})$$ and that $T=o(N)$, then, for $ N $ large enough,
\begin{align*}
\mathbb{P}(Z_1=Z_2+pT)&<\frac{S}{2\pi(1-p)N}\exp\left(-\frac{2pS^2}{(1-p)(2N-T)}+o(1)\right)\\
& {} +\frac{3}{\pi pS}\exp\left(-\frac{9pS^2}{8(1-p)N}\right),
\end{align*}
for independent random variables $Z_1=B(N-S,p)$ with mean $\mu_1$ and variance $\sigma_1^2$ and $Z_2=B(N+S-T,p)$ with mean $\mu_2$ and variance $\sigma_2^2$. 
\end{prop}

\begin{proof}
Let $\phi(i)$ be the probability that $Z_1=Z_2+pT=pN+i$, then 
$$\phi(i)=\binom{N-S}{pN+i}\binom{N+S-T}{pN-pT+i}p^{p(2N-T)+2i}(1-p)^{(1-p)(2N-T)-2i}.$$

\ind Denote the ratio between successive values of $\phi(i)$ by $\psi(i)$. We obtain
\begin{align*}
\psi(i)&=\frac{\phi(i+1)}{\phi(i)}=\frac{p^2((1-p)N-S-i)((1-p)(N-T)+S-i)}{(1-p)^2(pN+i+1)(p(N-T)+i+1)}.
\end{align*}
Hence, we get 
\begin{align}\label{psikvot}
\psi(i)&=\frac{\left(1-\frac{S+i}{(1-p)N}\right)\left(1+\frac{S-i}{(1-p)(N-T)}\right)}{\left(1+\frac{i+1}{pN}\right)\left(1+\frac{i+1}{p(N-T)}\right)},
\end{align}
and so $\psi$ is a decreasing function of $i$. By noting that $e^{x-x^2}\leq(1+x)\leq e^x$, for $x\geq -\frac{1}{2}$, we can bound $\psi$ for $i=o(p(1-p)N)$ (when $ N $ is large enough).
We apply $ e^{x-x^2}\leq(1+x) $ for the terms in the 
numerator of (\ref{psikvot}) and $ \leq(1+x)\leq e^x$ for the terms in the denominator of (\ref{psikvot}) to get the following lower bound of $ \psi $
\begin{align*}
\exp\left(\frac{pST-(2N-T)(i+1-p)}{p(1-p)N(N-T)}-\left(\frac{S+i}{(1-p)N}\right)^2-\left(\frac{S-i}{(1-p)(N-T)}\right)^2\right).
\end{align*}
and we apply $ \leq(1+x)\leq e^x$ for the terms in the 
numerator of (\ref{psikvot}) and  $ e^{x-x^2}\leq(1+x) $ for the terms in the denominator of (\ref{psikvot}) to get the following upper bound of $ \psi $
\begin{align*}\exp\left(\frac{pST-(2N-T)(i+1-p)}{p(1-p)N(N-T)}+\left(\frac{i+1}{pN}\right)^2+\left(\frac{i+1}{p(N-T)}\right)^2\right).
\end{align*}

\ind Substituting in $i=\pm\frac{pS}{2}$, we get (for $ N $ large enough) that 
\begin{align}\label{boundpsi1}
\psi(\frac{pS}{2})&<\exp\left(-\left(\frac{(2N-3T)S}{2(1-p)N(N-T)}\right)(1+o(1))\right)\nonumber\\
&<\exp\left(-\frac{(2N-3T)S}{4(1-p)N(N-T)}\right)\nonumber\\
&<1-\frac{S}{3(1-p)N}
\end{align}
and 
\begin{align}\label{boundpsi2}
\psi(-\frac{pS}{2})&>\exp\left(\left(\frac{(2N+T)S}{2(1-p)N(N-T)}\right)(1+o(1))\right)\nonumber\\
&>\exp\left(\frac{(2N+t)S}{4(1-p)N(N-T)}\right)\nonumber\\
&>1+\frac{S}{3(1-p)N}.
\end{align}
Therefore $\psi$ is greater than $1$ at $i=pN-\frac{pS}{2}$ and less than $1$ at $i=pN+\frac{pS}{2}$. Consequently (for $ N $ large enough), the maximum value of $\phi$ occurs between these two values.

\ind We have that
$$\phi(i)=\mathbb{P}(Z_1=\mu_1+pS+i)\mathbb{P}(Z_2^\prime=\mu_2^\prime+pS-i),$$
where  $$Z_2^\prime=N+S-T-Z_2=B(N+S-T,(1-p)),$$ with mean $\mu_2^\prime$ and variance $(\sigma_2^\prime)^2$. By Corollary~\ref{MajorBoot:binomuppercor} we get that 
$$\phi(i)<\frac{1}{2\pi \sigma_1\sigma_2^\prime}\exp\left(-\frac{(pS+i)^2(N+S-T)+(pS-i)^2(N-S)}{2p(1-p)(N-S)(N+S-T)}+o(1)\right),$$
for $|i|\leq\frac{pS}{2}$. This is maximized when $i=\frac{pST-2pS^2}{2N-T}$ and there takes the value
\begin{align*}
&\frac{1}{2\pi p(1-p)N}\exp\left(-\frac{pS^2((2N-T)^2-(T-2S)^2)}{2(1-p)(N-S)(N+S-T)(2N-T)}+o(1)\right)\\
&=\frac{1}{2\pi p(1-p)N}\exp\left(-\frac{2pS^2}{(1-p)(2N-T)}+o(1)\right).
\end{align*}

\ind We also obtain the bounds (for $ N $ large enough)
\begin{align*}
\phi(\frac{pS}{2})&<\frac{1}{2p(1-p)\pi N}\exp\left(-\frac{pS^2(10N+8S-9T)}{8(1-p)(N-S)(N+S-T)}+o(1)\right)\\
&<\frac{1}{2p(1-p)\pi N}\exp\left(-\frac{9pS^2}{8(1-p)N}\right)
\end{align*}
and
\begin{align*}
\phi(\frac{-pS}{2})&<\frac{1}{2p(1-p)\pi N}\exp\left(-\frac{pS^2(10N-8S-T)}{8(1-p)(N-S)(N+S-T)}+o(1)\right)\\
&<\frac{1}{2p(1-p)\pi N}\exp\left(-\frac{9pS^2}{8(1-p)N}\right).
\end{align*}

\ind Putting this all together and applying (\ref{boundpsi1}) and  (\ref{boundpsi2}), we obtain (for $ N $ large enough)
\begin{align*}
\mathbb{P}(Z_1=Z_2=pT)&<pS\max_i\phi(i)+\frac{1}{1-\psi(\frac{pS}{2})}\phi(\frac{pS}{2})+\frac{\psi(\frac{-pS}{2})}{\psi(\frac{-pS}{2})-1}\phi(\frac{-pS}{2})\\
&<\frac{S}{2\pi(1-p)N}\exp\left(-\frac{2pS^2}{(1-p)(2N-T)}+o(1)\right)\\
& {} +\frac{3}{p\pi S}\exp\left(-\frac{9pS^2}{8(1-p)N}\right).
\end{align*} 

\end{proof}

\ind We end with some propositions about the number of edges in and between sets in $G(n,p)$.

\begin{prop}
\label{MajorBoot:edgesinlarge}
Suppose that $p(1-p)n=\omega(n)$. If $n$ is large enough, then for all $t>\frac{n}{5},$ we have that with probability at least $1-4^{-t}$ every set in $G(n,p)$ of size $t$ has at most $p\binom{t}{2}+2t\sqrt{p(1-p)t}$ edges.
\end{prop}

\begin{proof}
The expected number of sets of size $t$ with more than   $$p\binom{t}{2}+2t\sqrt{p(1-p)t}$$ edges is
$$\binom{n}{t}\mathbb{P}\left(B\left(\binom{t}{2},p\right)\geq p\binom{t}{2}+2t\sqrt{p(1-p)t}\right).$$

\ind By Lemma~\ref{MajorBoot:Bernstein} and the fact that $\binom{n}{t}\leq (\frac{en}{t})^t$, this expectation is at most
$$(5e)^t\exp\left(-\frac{4p(1-p)t^3}{2p(1-p)\binom{t}{2}+\frac{4t\sqrt{p(1-p)t}}{3}}\right).$$

\ind As $\sqrt{p(1-p)t}=\omega(n)$, we have that if $n$ is large enough, then for all $t>\frac{n}{5}$ we have
$$2p(1-p)\binom{t}{2}+\frac{4t\sqrt{p(1-p)t}}{3}\leq 1.001p(1-p)t^2.$$

\ind Substituting this in we have that the expected number of sets of size $t$ with more than $p\binom{t}{2}+2t\sqrt{p(1-p)t}$ edges is (for $ n $ large enough) at most,
$$\exp\left(t(\log 5+1)-\frac{4p(1-p)t^3}{1.001 p(1-p)t^2}\right)<4^{-t}.$$

\end{proof}

%\begin{prop}
%\label{MajorBoot:edgesinsmall}
%Suppose that $p(1-p)n\geq3\log n$ and $t\leq \frac{n}{5}$ then with probability at least $1-n^{-2}$, every set in $G(n,p)$ of size $t$ has at most $(1+\frac{n}{2t})p\binom{t}{2}$ edges.
%\end{prop}

%\begin{proof}
%The expected number of sets of size $t$ with more than $(1+\frac{n}{2t})p\binom{t}{2}$ edges is

%$$\binom{n}{t}\mathbb{P}\left(B\left(\binom{t}{2},p\right)\geq (1+\frac{n}{2t})p\binom{t}{2}\right).$$

%By Lemma~\ref{MajorBoot:Bernstein} and the fact that $\binom{n}{t}\leq n^t$, this is at most

%$$n^t\exp\left(-\frac{(\frac{np\binom{t}{2}}{2t})^2}{2p(1-p)\binom{t}{2}+\frac{n}{3t}p\binom{t}{2}}\right)$$.

%Which is less than

%$$n^t\exp\left(-\frac{np(t-1)}{8}\right),$$

%as $2(1-p)\frac{n}{3t}<\frac{n}{t}$. As $p(1-p)n=\omega(n)\log n$ the above is certainly smaller than $n^{-2}$.  
%\end{proof}

\begin{prop}
\label{MajorBoot:edgesbetweenlarge}
Suppose that $p(1-p)n=\omega(n)$. If $n$ is large enough then for all $t$ in the range $\frac{n}{5}<t\leq\frac{n}{2},$ we have that with probability at least $1-4^{-t}$ every set in $G(n,p)$ of size $t$ has at least $pt(n-t)-3t\sqrt{p(1-p)(n-t)}$ edges between it and its complement.
\end{prop}

\begin{proof}
The expected number of sets $T$ of size $t$ with less than $pt(n-t)-3t\sqrt{p(1-p)(n-t)}$ edges between $T$ and $[n]\setminus T$ is
$$\binom{n}{t}\mathbb{P}\left(B(t(n-t),(1-p))\geq (1-p)t(n-t)+3t\sqrt{p(1-p)(n-t)}\right).$$

\ind By Lemma~\ref{MajorBoot:Bernstein} and the fact that $\binom{n}{t}\leq (\frac{en}{t})^t$, this expectation is at most
$$(5e)^t\exp\left(-\frac{9p(1-p)t^2(n-t)}{2p(1-p)t(n-t)+2t\sqrt{p(1-p)(n-t)}}\right).$$

\ind As $\sqrt{(n-t)p(1-p)}=\omega(n)$, we have that if $n$ is large enough, then for all $t$ in the range $\frac{n}{5}<t\leq\frac{n}{2}$,
$$2p(1-p)t(n-t)+2t\sqrt{p(1-p)(n-t)}\leq\frac{9}{4}p(1-p)t(n-t).$$

\ind Substituting this in we have that the expected number of sets $T$ with a small number of edges between $T$ and $[n]\setminus T$ is (for $ n $ large enough) less than
$$\exp\left(t(\log 5+1)-4t\right)<4^{-t}.$$
\end{proof}

\begin{prop}
\label{MajorBoot:edgesbetweensmall}
Suppose that $p(1-p)n\geq 4\log n$. If $n$ is large enough, then for all $t\leq \frac{n}{5}$ we have that with probability at least $1-n^{-\frac{t}{120}}$, for every set $T$ in $G(n,p)$ of size $t$ there are at least twice as many edges between $T$ and $[n]\setminus T$ as there are in $T$.
\end{prop}

\begin{proof}
The expected number of sets $T$ of size $t$ such that there are less than twice as many edges between $T$ and $[n]\setminus T$ as there are in $T$ is
$$\binom{n}{t}\mathbb{P}\left(B\left(t(n-t),p\right)< 2B\left(\binom{t}{2},p\right)\right),$$
for independent random variables $ B\left(t(n-t),p\right) $ and $ B\left(\binom{t}{2},p\right) $.

\ind We can rewrite this as,
$$\binom{n}{t}\mathbb{P}\left(2B\left(\binom{t}{2},p\right)-pt(t{-}1)-B(t(n{-}t),p)+pt(n{-}t)>pt(n{-}2t{+}1)\right).$$

\ind By Lemma~\ref{MajorBoot:Bernstein}, this is at most
\begin{equation}
\label{MajorBoot:edgessmalleqn}
\binom{n}{t}\exp\left(-\frac{(pt(n-2t+1))^2}{2p(1-p)t(n+t-2)+\frac{4pt(n-2t-1)}{3}}\right).
\end{equation}

\ind For $t<\frac{n}{24}$, using the inequality $\binom{n}{t}\leq n^t$ we have that the quantity in (\ref{MajorBoot:edgessmalleqn}) is (for $ n $ large enough) less than
$$n^t\exp\left(-\frac{pt(\frac{11n}{12})^2}{\frac{10n}{3}}\right)<
n^t\exp\left(-\frac{4t \log n}{n}\cdot\frac{121 n}{480}\right)=
n^{-\frac{t}{120}}.$$

\ind For $\frac{n}{24}\leq t\leq \frac{n}{5}$, using the inequality $\binom{n}{t}\leq\left(\frac{en}{t}\right)^t$ we have that the quantity in (\ref{MajorBoot:edgessmalleqn}) is (for $ n $ large enough) less than,
$$\left(\frac{en}{t}\right)^t\exp\left(-\frac{pt(\frac{3n}{5})^2}{\frac{10n}{3}}\right)<\left(\frac{24e}{n^{\frac{2}{5}}}\right)^t<n^{-\frac{t}{120}}.$$

\end{proof}

\begin{cor}
\label{MajorBoot:smallpmediumset}
Suppose that $pn\geq\log n$. If $n$ is large enough, then for all $t$ satisfying $n^{\frac{24}{25}}\leq t\leq\frac{n}{5}$, we have that with probability at least $1-n^{\frac{-t}{120}}$, for every set $T$ in $G(n,p)$ of size $t$, there are at least twice as many edges between $T$ and $[n]\setminus T$ than there are in $T$.
\end{cor}

\begin{proof}
By the exact same reasoning as in Proposition~\ref{MajorBoot:edgesbetweensmall} the expected number of sets $T$ of size $t$ with less than twice as many edges between $T$ and $[n]\setminus T$ than there are in $T$ is (for $ n $ large enough) at most
$$\left(\frac{en}{t}\right)^t\exp\left(-\frac{pt\left(\frac{3n}{5}\right)^2}{\frac{10n}{3}}\right)<\left(\frac{e}{n^{\frac{17}{250}}}\right)^t<n^{-\frac{t}{120}}.$$
\end{proof}

\begin{prop}
\label{MajorBoot:mindegree}
For every fixed $\epsilon>0$ and $p\geq \frac{(1+\epsilon)\log n}{n}$, with high probability, the minimal degree of $G(n,p)$ is greater than $8$.
\end{prop}

\begin{proof}
The expected number of vertices with degree at most $8$ is bounded by
\begin{align}
\label{MajorBoot:mindegbound}
&n\mathbb{P}(B(n-1,p)\leq 8)=n\sum_{i=0}^{8}\binom{n-1}{i}p^i(1-p)^{n-1-i}\nonumber\\
&\leq n\left(\binom{n-1}{8}p^8(1-p)^{n-9}\left(1+\frac{9(1-p)}{p(n-9)}+\left(\frac{9(1-p)}{p(n-9)}\right)^2+\ldots\right)\right)\nonumber\\
&\leq\frac{9n^{9}}{8!}p^8(1-p)^{n-9}.
\end{align}

\ind These inequalities follow as  $\max_{i\leq 8}\mathbb{P}(B(n-1,p)=i)$ occurs (for $ n $ large enough) when $i=8,$ and so $\mathbb{P}(B(n-1,p)\leq 8)\leq9\mathbb{P}(B(n-1,p)=8)$. The last line of (\ref{MajorBoot:mindegbound}) is maximised over $0\leq p\leq 1$ when $\frac{p}{8}=\frac{1-p}{n-9}$, that is when $p=\frac{8}{n-1}$. So for $p$ in our range, (\ref{MajorBoot:mindegbound}) is maximised when $p=\frac{(1+\epsilon)\log n}{n}$. Therefore (for $ n $ large enough)
\begin{align*}
n\mathbb{P}(B(n-1,p)\leq 8)&\leq \frac{9n^9(1+\epsilon)^8(\log n)^8}{8!n^8}e^{-\frac{(n-9)(1+\epsilon)\log n}{n}}\\
&\leq \frac{(\log n)^8}{n^{\frac{\epsilon}{2}}}.
\end{align*} 

\end{proof}

\begin{prop}
\label{MajorBoot:edgesinsmall}
Suppose that $(1+\epsilon)\log n\leq pn\leq 5\log n$. If $n$ is large enough, then for all $t$ satisfying $t\leq n^{\frac{29}{30}}$, we have that with probability at least $1-n^{-\frac{t}{120}}$, every set in $G(n,p)$ of size $t$ has at most $2t$ edges.
\end{prop}

\begin{proof}
The expected number of sets $T$ in $G(n,p)$ of size $t$ with at least $2t$ edges is
\begin{equation}
\label{MajorBoot:smallpedgebound}
\binom{n}{t}\mathbb{P}\left(B\left(\binom{t}{2},p\right)\geq 2t\right)=\binom{n}{t}\sum_{i=2t}^n\binom{\binom{t}{2}}{i}p^i(1-p)^{\binom{t}{2}-i}.
\end{equation}

\ind By carefully bounding the summands in (\ref{MajorBoot:smallpedgebound}) for $i=2t$ and $i=2t+1,$ we shall get a good bound on the total sum. We have that
$$\binom{\binom{t}{2}}{2t}p^{2t}(1-p)^{\binom{t}{2}-2t}<\left(\frac{ep(t-1)}{4}\right)^{2t}<\left(\frac{5et\log n}{4n}\right)^{2t}.$$

\ind We also get
$$\frac{\binom{\binom{t}{2}}{2t+1}p^{2t+1}(1-p)^{\binom{t}{2}-2t-1}}{\binom{\binom{t}{2}}{2t}p^{2t}(1-p)^{\binom{t}{2}-2t}}=\frac{p(\binom{t}{2}-2t)}{(1-p)(2t+1)}\leq pt<\frac{1}{2}.$$

\ind Because the ratio between consecutive terms in the sum in (\ref{MajorBoot:smallpedgebound}) decreases as $ i $ increases, we have from above that the total sum is at most twice the first term, therefore 
\begin{align*}
\binom{n}{t}\mathbb{P}\left(B\left(\binom{t}{2},p\right)\geq 2t\right)&\leq\binom{n}{t}2\left(\frac{5et\log n}{4n}\right)^{2t}\\
&\leq 2\frac{e^{3t}25^t(\log n)^{2t}t^t}{16^tn^t}\\
&\leq \left(\frac{C(\log n)^2}{n^{\frac{1}{30}}}\right)^t,
\end{align*}
and so the expected number of set $T$ in $G(n,p)$ of size $t$ with at least $2t$ edges is (for $ n $ large enough) at most $n^{-\frac{t}{120}}.$  
\end{proof}

\textbf{Acknowledgement.} We would like to express our gratitude to B. Bollob\'as and R. Morris who introduced us to the problem.

\bibliography{MajorBootbib}
\bibliographystyle{amsplain}

\end{document}